\documentclass[12pt]{article} %draft

\usepackage{amsmath, amsfonts, amssymb}
\usepackage{mathrsfs}
\usepackage{theorem}
\usepackage{bm}
\usepackage[dvipdf]{graphicx}

\RequirePackage[colorlinks,citecolor=blue,urlcolor=blue,bookmarksopen]{hyperref}

\newtheorem{mythm}{Theorem}[section]

\newtheorem{mylem}[mythm]{Lemma}

{\newtheorem{myrem}[mythm]{Remark}}
{}%\theorembodyfont{\rmfamily}
\def\R{\mathbb R}
\def\Z{\mathbb Z}
\def\N{\mathbb N}
\def\C{\mathscr C}
\def\B{\mathscr B}

\def\F{\mathscr F}
\def\d{\text{\rm{d}}}
\def\E{\mathbb E}
\def\p{\mathbb P}
\def\e{\text{\rm{e}}}

\def\la{\langle}
\def\raa{\rangle}
\def\La{\Lambda}
\def\veps{\varepsilon}

\def\S{\mathcal S}
\def\C{\mathscr C}

\def\pb{\mathscr{P}}

\def\wt{\widetilde}

\def\var{\mathrm{var}}
\def\W{\mathbb{W}}

\def\law{\mathcal{L}}

\newcommand{\fin}{\hfill $\square$\par}%{\hspace*{\fill}\rule{0.3em}{1ex}}
\newenvironment{proof}{{\noindent\it Proof.}\ }{\hfill $\square$\par}
\numberwithin{equation}{section}

\allowdisplaybreaks

\begin{document}
	
	\title{Conditional McKean-Vlasov SDEs with jumps and Markovian regime-switching: wellposedness, propagation of chaos, averaging principle \footnote{Supported in part by National Key R\&D Program of China (No. 2022YFA1000033) and NNSFs of China (No. 12271397,  11831014)}}
	%[Optimal Feedback control for stochastic systems]
	
	\author{Jinghai Shao${}^a$\thanks{a: Center for Applied Mathematics, Tianjin University, Tianjin, China.   }  \and Taoran  Tian${}^a$  \and Shen Wang${}^b$\thanks{b: College of Science, Civil Aviation University of China, Tianjin, China}
	}
	
\maketitle
	
\begin{abstract}
The conditional McKean-Vlasov SDEs with jumps and Markovian regime-switching are investigated in this work. We propose a new metric on the c\`adl\`ag path space to overcome the difficulty caused by the jumps  in order to establish the strong wellposedness of conditional McKean-Vlasov SDEs using $L^2$-Wasserstein distance. Also, we establish the propagation of chaos for the associated mean-field interaction particle system with common noise and provide an explicit bound on the convergence rate. Furthermore, an averaging principle is established for   two time-scale conditional McKean-Vlasov equations, where much attention is paid to the convergence of the conditional distribution term.
\end{abstract}
	
\section{Introduction}
Conditional McKean-Vlasov SDEs can be used to describe the asymptotic behavior for a class of $N$-particle systems with mean-field interaction in a common environment characterized by a regime-switching process. The coefficients of these type of SDEs depend on the solution of the equation and the conditional law of this solution with regime-switching. The mean-field interactions can be expressed by the coefficients depending on the empirical measure of the $N$-particle system, which converges to the conditional law of any particle of the limit system as $N$ goes to infinity. There is an extensive list of references to such works, see \cite{CDL} for mean-field games with common noise, \cite{Szn} for a nice introduction to the topic of propagation of chaos, \cite{Gar,Oel} for the study of the McKean-Vlasov limit using martingale method, \cite{DZ} for some limit theory for jump mean-field models, \cite{APF,Gra,GM} for the study of jump diffusion models and the references therein.  Mean-field models have arisen in various of application areas, see \cite{BFFT} for neural network modeling, \cite{BFY} for mean-field games and
	 mean-field type control theory, \cite{CGM} for mean-field theory in social sciences and \cite{FL} for an application in portfolio modeling. The purpose of this work is to study the existence and uniqueness of solutions to conditional McKean-Vlasov (MKV) stochastic differential equations (SDEs) driven by L\'evy processes with Markovian regime-switching, and the conditional propagation of chaos for an associated mean-field interaction particle system with common noise. Moreover, we shall investigate the averaging principle for two time-scale MKV SDEs with Markovian regime-switching.

	 Precisely, let $(\Omega^0,\F^0,\F_t^0,\p^0)$ and $(\Omega^1,\F^1,\F_t^1,\p^1)$ be two filtered probability spaces satisfying $\F_{t}^k=\F_{t+}^k:=\bigcap_{s\geq t} \F_{s}^k$, $k=0,1$. Define a product space
	 \[\Omega=\Omega^0\times\Omega^1,\quad \p=\p^0\times\p^1,\]
	 and $\F$ stands for the completion of $\F^0\times\F^1$, $\F_t$ the completion of $\F_t^0\times\F_t^1$ for $t\geq 0$. Put
	 \[\mathbb{F}=(\F_t)_{t\geq 0},\quad \mathbb{F}^0=(\F_t^0)_{t\geq 0}, \quad \mathbb{F}^1=(\F_t^1)_{t\geq 0}.
	 \]
	 Element $\omega=(\omega^0,\omega^1)\in \Omega^0\times\Omega^1$ denotes a generic element in $\Omega$. $\E^0$ and $\E^1$ stands for taking expectation w.r.t.\! $\p^0$ and $\p^1$ respectively. Given a random variable $X$, denote by $\law(X)$  the law of $X$, and  by $\law(X|\hat \F)$ the conditional law of $X$ given a $\sigma$-algebra $\hat \F$.
	 Consider the following conditional MKV SDE:
	 \begin{equation}\label{a-1}
	 	\begin{split}
	 		\d Y_t&= b(Y_t,\law(Y_t|\F_{t}^{0}), \La_t)\d t+\sigma(Y_t,\law(Y_t|\F_{t}^{0}),\La_t)\d W_t \\
	 		&\qquad +\int_{\R^d_0}\!g(Y_{t-},\law(Y_{t-}|\F_{t-}^{0}), \La_{t-},z)\wt{\mathcal{N}}(\d t,\d z),
	 	\end{split}
	 \end{equation}
	 where $b:\R^d\times\pb(\R^d)\times\S\to \R^d$, $\sigma:\R^d\times\pb(\R^d)\times\S\to \R^{d\times d}$, $g:\R^d\times\pb(\R^d)\times \S\times \R^d\to \R^d$ satisfying conditions specified later, $\F_{t-}^0:=\bigcup_{s< t} \F_{s}^0$. $\pb(\R^d)$ denotes the space of all probability measures over $\R^d$. $\S=\{1,2,\ldots,m_0\}$ and  $m_0 $ could equal $\infty$, which means that $\S$ is an infinitely countable state space when $m_0=\infty$. $(W_t)$ is $d$-dimensional Brownian motion, $\wt{\mathcal{N}}(\d t,\d z)=\mathcal{N}(\d t,\d z)-\lambda(\d z)\d t$ is a compensated Poisson random measure with intensity $\lambda(\d z)\d t$ satisfying $\int_{\R^d_0}1\!\wedge |z|^2\lambda(\d z)<\infty$ and $\R^d_0:=\R^d\backslash\{0\}$. $(\La_t)$ is a continuous-time Markov chain on $\S$ with transition rate matrix $(q_{ij})_{i,j\in\S}$. Throughout this work, suppose that $(W_t)$, $(\wt{\mathcal{N}}(\d t,\d z))$ and $(\La_t)$ are mutually independent, and further that Brownian motions and Poisson random measures are all defined on the probability space $(\Omega^1,\F^1,\p^1)$ and Markov chains are defined on $(\Omega^0,\F^0,\p^0)$ in order to simplify the treatment  of conditional distribution in  SDE \eqref{a-1} adopting the basic setting of Carmona and Delarue \cite{Car} in the study of MKV SDEs with common noise.

	 In this paper we shall study the strong wellposedness and the conditional  propagation of chaos for the conditional MKV SDE \eqref{a-1} and the associated mean-field interaction particle system with a common noise characterized by a Markovian regime-switching process. These questions for MKV SDEs driven by L\'evy process without regime-switching  have been studied by \cite{Gra} and the recent work \cite{Er}. Also, notice that in \cite{Gra,Er} the distribution instead of the conditional distribution as in \eqref{a-1} was used. The wellposedness of \eqref{a-1} without the term $\wt{\mathcal{N}}(\d t,\d z)$ has been studied in \cite{NYH}, which generalized the result in \cite{Oel} without regime-switching. \cite{NYH} also showed the law of large number of a mean-field interaction system with Markovian regime-switching, which deduces the conditional distribution $\law(Y_t|\F_t^0)$ instead of the distribution $\law(Y_t)$ to be used in the presence of regime-switching. Later, \cite{SW21} generalized the work \cite{NYH} to conditional MKV SDEs driven by Brownian motion with state-dependent regime-switching. Comprehensive study on hybrid switching diffusion processes can be found in \cite{YZ} and the references therein. The weak convergence or $L^1$-convergence of the particle system to its limit has been established in \cite{Er,Gra,NYH}, whereas there is no estimate on the convergence rate. In this work, we shall establish the $L^2$-convergence and provide an  estimate of the convergence rate.
	
	 In the process of solving the previous two questions, we are interested in overcoming a technical challenge: on $\pb(\R^d)$,
	 use the $L^2$-Wasserstein distance to measure the regularity of coefficients such as $\mu\mapsto b(x,\mu,i)$ rather than using the bounded Lipschitz distance as in \cite{NYH}, or using the $L^1$-Wasserstein distance as in \cite{Er,Gra}. We prefer using $L^2$-Wasserstein distance for two reasons: 1) the Lipschitz continuous condition on coefficients $b$, $\sigma$, $g$ in terms of $L^2$-Wasserstein distance is weaker than  those in terms of bounded Lipschitz distance or $L^1$-Wasserstein distance (see the details in Section 2). 2) The probability space $\pb(\R^d)$ endowed with $L^2$-Wasserstein distance owns a rich geometric structure, and many known results for MKV SDEs driven by Brownian motions are presented using $L^2$-Wasserstein distance (cf. \cite{Car} and references therein).
	
	 Our second purpose is to study the averaging principle for two time-scale conditional MKV SDEs with regime-switching. We study the the averaging principle in the $L^2$-convergence sense, therefore, more restrictions on the coefficients $\sigma$ and $g$ are needed. A little precisely, for $\veps>0$, in this part we consider the two time-scale system $(Y_t^\veps,\La_t^\veps)$ satisfying
\begin{equation}\label{a-2}
 \begin{split}
	 		\d Y_t^\veps&= b(Y_t^\veps,\law(Y_t^\veps|\F_{t}^{0}), \La_t^\veps)\d t+\sigma(Y_t^\veps,\law(Y_t^\veps|\F_{t}^{0}))\d W_t  +\int_{\R^d_0}\!g( z)\wt{\mathcal{N}}(\d t,\d z),
	 	\end{split}
	 \end{equation}
where $(\La_t^\veps)$ is a Markov chain on $\S$ with transition rate matrix $\frac1\veps(q_{ij})_{i,j\in\S}$. In this two time-scale system,  $(Y_t^\veps)$ is the slow component, and $(\La_t^\veps)$ is the fast component. Under certain ergodic condition for the semigroup $P_t$ associated with $Q$-matrix $(q_{ij})_{i,j\in\S}$, we shall show that $(Y_t^\veps)$ will converge in $L^2$-norm to a limit process, which is a solution to MKV SDE with the coefficients to be the average of $b$, w.r.t.\! the invariant probability measure of $P_t$. Especially, the conditional distribution  $\law(Y_t^\veps|\F_t^0)$ will change to the distribution $\law(\bar{Y}_t)$; see the details in Section 3. During this procedure, as a characteristic of conditional MKV SDE, much attention is paid to determining the limit of the conditional distribution  $\law(Y_t^\veps|\F_t^{0})$ as $\veps\to 0$. When $\S$ is a finite state space, if the random switching process itself owns two time-scale too, i.e. $Q^\veps=\frac1{\veps} \tilde Q+\hat{Q}$, \cite{NYH} showed the weak convergence of the mean-field interaction particle system.
	
	 The rest of this work is organized as follows. In Section 2, we use the decoupling method and fixed point theorem to establish the wellposedness of conditional MKV SDE \eqref{a-1}. In Section 3, an averaging principle for a two time-scale conditional MKV SDE with regime-switching. In particular, we can deal with the switching process in an infinitely countable state space under certain ergodic condition on the Markov chain $(q_{ij})_{i,j\in\S}$. Section 4 is devoted to the study of conditional propagation of chaos for the mean-field interaction particle system associated with \eqref{a-1}. The strong convergence rate is explicitly estimated.

	 \section{Wellposedness of conditional MKV equation with regime-switching}
	
	 We begin with introducing some notation used in this work. Given a metric space $(E,d)$, denote by $\pb(E)$ the space of all probability measures on $(E,\B(E))$. For $p\geq 1$, put
	 \[\pb_p(E)=\big\{\mu\in \pb(E); \int_{E}\!d(x_0,x)^p\mu(\d x)<\infty\big\}\]
	 for some $x_0\in E$. For any two $\mu,\nu\in \pb(E)$, the $L^p$-Wasserstein distance between them is defined by
	 \begin{equation}\label{c-1}
	 	\W_p(\mu,\nu)=\inf\Big\{\int_{E\!\times\!E}\!d(x,y)^p\Gamma(\d x,\d y);\ \Gamma\in \C(\mu,\nu)\Big\}^{1/p},
	 \end{equation}
	 where $\C(\mu,\nu)$ denotes the set of probability measures over $E\!\times \!E$ with marginals $\mu$ and $\nu$. From this definition, it is easy to see  that
	 \[ \W_p(\mu,\nu)\leq \W_q(\mu,\nu),\quad \text{if $1\leq p\leq q$.}\]
	 The total variation distance between $\mu$ and $\nu$ is defined by
	 \[\|\mu-\nu\|_\var=2\sup\big\{|\mu(A)\!-\!\nu(A)|; A\in \! \B(E)\big\}=\sup\Big\{\big|\!\int_E\! f\d \mu\!-\!\int_E\! f\d \nu\big|; |f|\leq 1\Big\}.
	 \]

	 The following conditions on the coefficients will be used in this work.
	 \begin{itemize}
	 	\item[$\mathrm{(A1)}$] There exists $K_1>0$ such that for any $x,y\in \R^d$, $\mu,\nu\in \pb(\R^d)$, $i\in\S$,
	 	\begin{align*}
	 		&|b(x,\mu,i)-b(y,\nu,i)|^2+\|\sigma(x,\mu,i)-\sigma(y,\nu,i)\|^2 \\ &\ +\int_{\R^d_0}\! |g(x,\mu,i,z)-g(y,\nu,i,z)|^2\lambda(\d z)\leq K_1\big(|x-y|^2+\W_2^2(\mu,\nu)\big).
	 	\end{align*}
	 	\item[$\mathrm{(A2)}$] There exists $K_2>0$ such that for any $x\in \R^d$, $\mu\in \pb(\R^d)$, $i\in\S$,
	 	\begin{align*}
	 		&|b(x,\mu,i)|^2\!+\!\|\sigma(x,\mu,i)\|^2 \!+\! \int_{\R^d_0}\!|g(x,\mu, i,z)|^2\lambda(\d z)\leq K_2\big(1\!+\!|x|^2\!+\!\int_{\R^d_0}\!\!|y|^2\mu(\d y)\big).
	 	\end{align*}
	 	\item[$\mathrm{(A3)}$] Let $(q_{ij})_{i,j\in \S}$ be a conservative, irreducible transition rate matrix. When $\S=\{1,2,\ldots\}$ is infinitely countable, assume that there exist a function $V:\S\to [1,\infty)$ and a constant $\kappa_0\in \R$ such that
	 	\[\lim_{i\to \infty} V_i=\infty,\quad \sum_{j\in\S} q_{ij}V_j\leq \kappa_0V_i,\quad \forall \,i\in\S.
	 	\]
	 	\item[$\mathrm{(A4)}$] Let $P_t$ be the semigroup associated with $(q_{ij})_{i,j\in\S}$. Assume $P_t$ is ergodic and $\pi=(\pi_i)_{i\in\S}$ denotes its unique invariant probability measure. Suppose that there exist    constants   $\eta,\beta >0$
such that
\begin{equation}\label{c-2}
\begin{split}
 \|P_t(i,\cdot)\!-\!\pi\|_{\var}&\leq  \eta\, \e^{-\beta t},\ \quad \forall\,i\in\S,
 \end{split}
\end{equation}
	 \end{itemize}
	
	 \begin{myrem}\label{rem-1}
	 	Conditions (A1) and (A2) are used to ensure the wellposedness of MKV equations. In (A3), the existence of function $V$ is used to ensure the existence and uniqueness of the $Q$-process $P_t$ associated with $(q_{ij})$; see, for instance, \cite[Theorem 2.25]{Chen}. (A4) is used in the study of averaging principle for the interaction particle systems or conditional MKV SDEs in order to deal with the case $\S$ is infinitely countable. When $\S$ is a finite state space, it is well known that the semigroup $P_t$ satisfies condition \eqref{c-2}.
	 \end{myrem}
	
	 %\begin{myexam}\label{exam-1}
	 %\begin{enumerate}
	 %  \item provide a Lyapunov criterion such that (A4) holds.
	 %  \item Provide some concrete examples of $(\La_t)$ satisfying (A4).
	 %\end{enumerate}
	 %\end{myexam}
	
	 On the probability space $(\Omega, \F,\p)$ introduced above, let $(W_t)$ be $d$-dimensional standard Brownian motion, and $\wt{\mathcal{N}}(\d t,\d z)$ be a compensated Poisson random measure with intensity $\lambda(\d z)\d t$ satisfying $\int_{\R^d_0}1\wedge |z|^2\lambda(\d z)<\infty$ defined on $(\Omega^1,\F^1,\p^1)$. Assume as usual $(W_t)$ and $(\wt{\mathcal{N}}(\d z,\d t))$ are mutually independent. Let $(\La_t)$ be a continuous-time Markov chain with $Q$-matrix $(q_{ij})$ defined on $(\Omega^0,\F^0,\p^0)$.
	
	 Consider the following conditional MKV SDE with Markovian regime-switching:
	 \begin{equation}\label{c-3}
	 	\begin{split}
	 		\d Y_t&=b(Y_t,\law(Y_t|\F^0_t),\La_t)\d t+\sigma(Y_t,\law(Y_t|\F_t^0),\La_t)\d W_t\\
	 		&\qquad \qquad+\!\int_{\R^d_0}\!\! g( Y_{t-},\law(Y_{t-}|\F_{t-}^0), \La_{t-}, z)\wt{\mathcal{N}}(\d t,\d z),
	 	\end{split}
	 \end{equation}
	 with initial value $Y_0\in \F_0$ satisfying $\E|Y_0|^2<\infty$ and $\La_0=i_0\in \S$.

\begin{mythm}[Strong uniqueness]\label{thm-2.1}
	 	Under conditions (A1)-(A3), the strong uniqueness holds for the MKV SDE \eqref{c-3}, that is, for any two solutions $(Y_t)$ and $(\wt Y_t)$ to \eqref{a-1}, if $Y_0=\wt Y_0$ a.e. with $\E|Y_0|^2<\infty$, then for any $T>0$
	 	\[Y_t=\wt Y_t,\quad \forall \, t\in [0,T],\  \text{a.e.}.\]
	 \end{mythm}
	
	 \begin{proof}
	 	As $Y_t$ and $\wt Y_t$ satisfy  \eqref{a-1}, we have
	 	\begin{align*}
	 		|Y_t-\wt Y_t|^2&\leq 3\Big|\int_0^t\!(b(Y_s,\law(Y_s|\F_s^0),\La_s)-b(\wt Y_s,\law(\wt Y_s|\F_s^0),\La_s)\d s\Big|^2\\
	 		&\quad +\!3\Big|\int_0^t\!(\sigma(Y_s,\law(Y_s|\F_s^0),\La_s)- \sigma(\wt Y_s,\law(\wt Y_s|\F_s^0),\La_s)\d W_s\Big|^2\\
	 		&\quad + 3\Big|\int_0^t\!\int_{\R_0^d}\!\!\big(g(Y_{s-}, \law(Y_{s-}|\F_{s-}^0),\La_{s-},z)-\! g(\wt Y_{s-},\law(\wt Y_{s-}|\F_{s-}^0),\La_{s-},z)\big) \wt{\mathcal{N}}(\d s,\d z)\Big|^2
	 	\end{align*}
	 	Notice that for all $\omega^0\in \Omega^0$, $[\law(Y_s|\F_s^0)(\omega^0)](D)=\p^1[Y_s(\omega^0,\cdot)\in D]$, $D\in \B(\R^d)$, and hence by the definition of Wasserstein distance it holds
	 	\[\E\big[\W_2(\law(Y_s|\F_s^0),\law(\wt Y_s|\F_s^0))^2\big]\leq \E\big[\E^1[ |Y_s-\wt Y_s|^2]\big]=\E\big[|Y_s-\wt Y_s|^2\big].\]
	 	Applying Burkholder-Davis-Gundy's inequality and (A1), we get
\begin{equation}\label{c-3.1}
\E\big[\sup_{0\leq t\leq T}|Y_t-\wt Y_t|^2\big]\leq C(T)\E\Big[\int_0^T|Y_s-\wt Y_s|^2\d s\Big],
\end{equation} where $C(T)=6(TK_1+2C_2K_1)$, $C_2$ is a positive constant derived from Burkholder-Davis-Gundy's inequality. Then, the desired conclusion follows immediately by Gronwall's inequality.
	 \end{proof}

	 We have proven the strong uniqueness of solutions to \eqref{a-1} in Theorem \ref{thm-2.1}, then Yamada-Watanabe principle (cf. \cite{IW,Kurtz}) allows to construct a strong solution based on a weak solution of SDE \eqref{a-1}. In order to show the existence of weak solution, Erny \cite{Er} used Picard iteration method based on a local Lipschitz condition on the coefficients without regime-switching. There, the continuity of the coefficients w.r.t. the argument $\mu\in \pb(\R^d)$ is measured using the $L^1$-Wasserstein distance $\W_1$. \cite{Oel} used the convergence of mean-field interacting particle system to construct the weak solution to MKV SDEs, whose results have been generalized in \cite{NYH} to the setting of  MKV SDEs with Markovian regime-switching. \cite{Oel} considered MKV SDEs driven by Brownian motions and Poisson random measures respectively, and \cite{NYH} only considered the case driven by Brownian motions. In contrast to \cite{Er}, \cite{Oel} and \cite{NYH} adopted the bounded Lipschitz metric $\|\,\cdot\,\|_{BL}$ to measure the regularity of the coefficients w.r.t. \!$\mu\!\in \!\pb(\R^d)$, which is stronger than using $\W_1$. Precisely,
	 \[\|\mu-\nu\|_{BL}:=\sup\Big\{ \int_{\R^d}\! f\d \mu-\int_{\R^d}\! f\d \nu;\ |f|\leq 1,\sup_{x\neq y}\frac{|f(x)-f(y)|}{|x-y|}\leq 1\Big\}.\]
	 According to Kantorovich dual representation theorem for $L^p$-Wasserstein distance (cf. \cite{Vill}),
	 \[\W_1(\mu,\nu)=\sup\Big\{\int_{\R^d}\! f\d \mu-\int_{\R^d}\! f\d \nu; \ \sup_{x\neq y}\frac{|f(x)-f(y)|}{|x-y|}\leq 1\Big\}.\]
	 This yields immediately that
	 \[\|\mu-\nu\|_{BL}\leq \W_1(\mu,\nu).\]
	 However, in the study of MKV SDEs driven by Brownian motions, a widely used distance in $\pb(\R^d)$ is $L^2$-Wasserstein distance (cf. \cite{Car} and references therein). In what follows, we aim to solve MKV SDEs driven by L\'evy processes by using $L^2$-Wasserstein distance $\W_2$ on $\pb(\R^d)$. To achieve this purpose, the price we pay is that the jump noise becomes an additive noise in order to cope with the difficulty caused by the Skorokhod topology. A little precisely, the coefficient $g$ in \eqref{a-1} before the Poisson random measure cannot depend on $Y_t$ and its distribution.
	
	 Before presenting our result, we introduce some notation.
	 Let $\mathcal{C}([0,T];\R^d)=\{f:[0,T]\to \R^d;\ \text{being continuous}\}$.  Denote by $\mathcal{D}([0,T]; \R^d )$ the space of  right-continuous functions taking values in $ \R^d $ with left limits.
	 For any two curves $ \xi,\tilde \xi\in \mathcal{D}([0,T]; \R^d)$, define their distance by
	 \begin{equation}\label{c-4}
	 	\rho_{\mathcal{D}}( \xi, \tilde \xi)=\inf_{\gamma\in \Upsilon}\Big\{\max\big\{\|\gamma\|,\sup_{t\in [0,T]} \big|\xi_t-\tilde \xi_{\gamma(t)}\big|\big\}\Big\},
	 \end{equation}
	 where $\Upsilon$ is the class of continuous, strictly increasing maps acting from $[0,T]$ onto itself, and for $\gamma \in \Upsilon$,
	 \[\|\gamma\|:=\sup_{s\neq t}\Big|\!\log \frac{\gamma(t)-\gamma(s)}{t-s} \Big|.\]
	 By virtue of \cite[Theorems 14.1, 14.2]{Bill}, $\mathcal{D}([0,T];\R^d)$ becomes a Polish space under the distance $\rho_{\mathcal{D}}$.

	 For any given $T>0$, let $\mathscr{Y}_T$ be a space consisting of all processes $X:[0,T]\times\Omega\to  \R^d $ satisfying
	 \begin{itemize}
	 	\item[$\mathrm{(i)}$] for almost all $\omega \in \Omega $, $X_\cdot(\omega ):[0,T]\to  \R^d $ is right-continuous with left limits;
	 	\item[$\mathrm{(ii)}$] the  process $t\mapsto X_t$ is $\F_t$-adapted and satisfies $\sup_{t\in [0,T]}\E[|X_t|^2]<\infty$.
	 \end{itemize}
	 Based on the metric $\rho_{\mathcal{D}}$ on $\mathcal{D}([0,T];\R^d)$,  for each $p\geq 1$,  introduce  a metric on $\mathscr{Y}_T$:
	 \begin{equation}\label{c-5}
	 	{\bm\rho}_p(X,\tilde X):=\E\big[\rho_{\mathcal{D}} (X,\tilde X)^p\big]^{1/p},\quad X,\tilde X\in \mathscr{Y}_T.
	 \end{equation}
	 %Here, with a slight abuse  of notation, we use $X$, $\tilde X$ to denote both stochastic processes and their sample paths in the path space $\mathcal{D}([0,T];\R^d)$.
	
	 For a given element $\ell\in \mathscr{Y}_T$, define a subspace $\mathscr{Y}_T(\ell)$ of $\mathscr{Y}_T$ as follows:
	 \begin{equation}\label{c-4.5}
	 	\mathscr{Y}_T(\ell)=\big\{X\in \mathscr{Y}_T; \ \text{such that $X_\cdot-\ell_\cdot\in \mathcal{C}([0,T];\R^d)$}\big\}.
	 \end{equation}On $\mathscr{Y}_T (\ell)$, we use the distance induced from uniform norm of $\mathcal{C}([0,T];\R^d)$, that is,
	 \begin{equation}\label{c-4.6}
	 	{\bm\rho}_{p,\infty}(X,\tilde X):=\E\big[\|X_t-\tilde X_t\|_{\infty}^p\big]^{1/p}=\E\big[\sup_{t\in[0,T]}|X_t-\tilde X_t|^p\big]^{1/p},\quad X,\,\tilde X\in \mathscr{Y}_T(\ell).
	 \end{equation}

	 \begin{mylem}\label{lem-2.1}
	 	$\mathrm{(i)}$ The metric space $(\mathscr{Y}_T,{\bm \rho}_p)$ is complete for every $p\geq 1$.
	 	
	 	$\mathrm{(ii)}$ For each given $\ell\in \mathscr{Y}_T$, the metric space $(\mathscr{Y}_T(\ell), {\bm\rho}_{Tp,\infty})$ is complete for every $p\geq 1$.
	 \end{mylem}
	
	 \begin{proof}
	 	$\mathrm{(i)}$\ Let $(X^{(n)})_{n\geq 1}$ be a Cauchy sequence in $(\mathscr{Y}_T,{\bm \rho}_p)$, i.e.
	 	\[\lim_{n,m\to \infty} {\bm \rho}_p(X^{(n)},X^{(m)})=0.\]
	 	Take a subsequence of $(X^{(n)})$ such that
	 	\[{\bm \rho}_{p}(X^{(n_k)},X^{(n_{k+1})})<\frac 1{2^k},\quad k\geq 1,\] which yields that
	 	\[\p\big(\rho_{\mathcal{D}}(X^{(n_k)},X^{(n_{k+1})})\geq \frac 1{k^2}\big)\leq \frac{k^{2p}}{2^{kp}},\quad k\geq 1.\]
	 	Since $\sum_{k\geq 1}\frac{k^{2p}}{2^{kp}}<\infty$, we get from Borel-Cantelli's lemma that
	 	\[\p\big(\rho_{\mathcal{D}}(X^{(n_k)},X^{(n_{k+1})})\geq \frac 1{k^2},\ \text{i.o.}\big)=0.\]
	 	Hence, there is a measurable set $\Xi$ with $\p^0(\Xi)=1$ such that for every $\omega\in \Xi$
	 	\[\sum_{k=1}^\infty \rho_{\mathcal{D}}\big(X^{(n_k)}(\omega), X^{(n_{k+1})}(\omega)\big)<\infty.\]
	 	By the completeness of $\big(\mathcal{D}([0,T];\R^d), \rho_{\mathcal{D}}\big)$, there exists a $\tilde X(\omega)\in \mathcal{D}([0,T];\R^d)$ such that
	 	\begin{equation}\label{c-6}
	 		\lim_{k\to \infty} \rho_{\mathcal{D}}(X^{(n_k)}(\omega), \tilde X(\omega))=0,\quad \omega\in \Xi.
	 	\end{equation}

	 	Next, we shall show $\tilde X\in \mathscr{Y}_T$ and ${\bm \rho}_p(X^{(n_k)},\tilde X)\to 0$ as $k\to \infty$. For the first assertion, as $\tilde X(\omega)\in \mathcal{D}([0,T]; \R^d)$,  it suffices to show $\tilde X$ is $\F_t$-adapted, which is the key point of this lemma. To this aim, it is crucial to apply the following property of a probability measure over $\mathcal{D}([0,T];\R^d)$. Let $\tilde \Gamma$ be the distribution of $\tilde X$ in $\mathcal{D}([0,T];\R^d)$, and let
	 	$\mathscr{T}_{\tilde \Gamma}$ consist of those $t\in [0,T]$ such that $\p(J_t)=0$, where
	 	\[J_t=\{X\in \mathcal{D}([0,T];\R^d);\ X(t)\neq X(t-)\}.\] According to \cite[p.124]{Bill}, $\mathscr{T}_{\tilde \Gamma}$ contains $0$ and $T$ and its complement in $[0,T]$ is at most countable. As Skorokhod convergence implies $X_t^{(n_k)}\to \tilde X_t$ for continuity points $t$ of $\tilde X$, it yields from \eqref{c-6} that
	 	\begin{equation}\label{c-7}
	 		\lim_{k\to \infty} |X_t^{(n_k)}-\tilde X_t|=0,\quad \p\text{-a.e.}, \ \forall\,t\in \mathscr{T}_{\tilde \Gamma}.
	 	\end{equation}
	 	Combining this with the fact $X_t^{(n_k)}$ is $\F_t$-adapted, we have that $\tilde X_t$ is $\F_t$-adapted for all $t\in \mathscr{T}_{\tilde \Gamma}$. Furthermore, since   $t\mapsto \tilde X_t$ is right-continuous, for any $t\in (0,T)$ there exists a sequence of $t_n\in \mathscr{T}_{\tilde \Gamma}\downarrow t$ such that
	 	$\tilde X_{t_n}\longrightarrow \tilde X_t$, $\p$-a.e. as $n\to \infty$. By $\tilde X_{t_n}\in \F_{t_n}$ and $\F_{t}=\F_{t+}$, we finally get that $\tilde X_t\in \F_t$ for all $t\in [0,T]$. Therefore, $\tilde X\in \mathscr{Y}_T$.
	 	
	 	Due to Fatou's lemma, by \eqref{c-6} and $(X^{(n)})_{n\geq 1}$ is a Cauchy sequence in $\mathscr{Y}_T$, we have
	 	\[ \E\big[\rho_{\mathcal{D}}(X^{(n)},\tilde X)^p\big]\leq \lim_{l\to \infty} \E\big[\rho_{\mathcal{D}}(X^{(n)},X^{(n_l)})^p\big],
	 	\]
	 	and further
	 	\[\lim_{n\to \infty}{\bm \rho}_p(X^{(n)},\tilde X)\leq \lim_{n\to \infty}\E\big[\rho_{\mathcal{D}} (X^{(n)},\tilde X)^p\big]^{1/p}\leq \lim_{n\to \infty} \lim_{l\to \infty} \E\big[\rho_{\mathcal{D}}(X^{(n)},X^{(n_l)})^p\big]^{1/p} =0.
	 	\]
	 	It follows immediately that $X^{(n)}$ converges to $\tilde X$ in $\mathscr{Y}_T$, and we conclude that $(\mathscr{Y}_T,{\bm \rho}_T)$ is a complete metric space.
	 	
	 	$\mathrm{(ii)}$\ Let $(X^{(n)})_{n\geq 1}$ be a Cauchy sequence in $(\mathscr{Y}_T(\ell),{\bm\rho}_{p,\infty})$, which implies that
	 	\[ \E\big[\|X^{(n)}-X^{(m)}\|_\infty^p\big]= {\bm \rho}_{p,\infty}(X^{(n)}, X^{(m)})^p\longrightarrow 0,\ \ \text{as $n,m\to \infty$.}
	 	\]
	 	Consider the sequence of  random functions $(X^{(n)}-\ell)_{n\geq 1}$ in $\mathcal{C}([0,T];\R^d)$.
	 	Then following the approach of (i) we can show that there exists a subsequence  $(X^{(n_k)}-\ell)_{k\geq 1}$ and a random function  $(\tilde Y_t)_{t\in [0,T]}$ in $\mathcal{C}([0,T];\R^d)$ such that
	 	\[\lim_{k\to \infty} \|X^{(n_k)}(\omega)-\ell(\omega)-\tilde Y(\omega)\|_{\infty} =0,\quad \text{$\p$-a.e. $\omega$.}\]
	 	Hence, we can get from the $\F_t$-adaptness of $X^{(n)}$ and $\ell$ that  $(\tilde Y_t)$ is also $\F_t$-adapted, and further $\tilde X_t:=\tilde Y_t+\ell_t$, $t\in [0,T]$, is the desired limit point of $(X^{(n)})_{n\geq 1}$ in $(\mathscr{Y}_{T}(\ell), {\bm \rho}_{p,\infty})$.
	 \end{proof}

	 \begin{mythm}\label{thm-2.2}
	 	Assume  that (A1)-(A3) hold, and in addition that the coefficient $g(x,\mu,i,z)$ does not depend on arguments $x$ and $\mu$.
	 	Then, MKV SDE \eqref{a-1} admits a unique strong solution.
	 \end{mythm}
	
	 \begin{proof} We prove this theorem using the decoupling method, which has been widely used in the study of MKV SDEs; see the monograph \cite[Chapter 2]{Car} and references therein for its application in the study of MKV SDEs with common noise. To prove this theorem, it suffices to show that for each fixed $T>0$ there exists a unique strong solution $(Y_t)_{t\in [0,T]}$ for SDE \eqref{a-1}.
	 	
	 	Let $\ell_t=\int_0^t\!\int_{\R_0^d}\!g(\La_s,z)\wt{\mathcal{N}}(\d s, \d z)$ for $t\in [0,T]$. It is clear that $\ell\in \mathscr{Y}_T$.
	 	For a given $X\in \mathscr{Y}_T(\ell)$, denote  $\mu_t=\law(X_t|\F_t^0)$. Consider the SDE with regime-switching:
	 	\begin{equation}\label{c-8}
	 		\begin{split}
	 			\d Y_t^\mu&=b(Y_t^\mu, \mu_t,\La_t)\d t+\sigma(Y_t^\mu,\mu_t,\La_t)\d W_t+\d \ell_t
	 		\end{split}
	 	\end{equation}
	 	with $Y_0^\mu=Y_0$. Under the conditions (A1)-(A3), according to \cite{MY,Xi}, SDE \eqref{c-8} admits a unique strong solution $(Y_t^\mu)_{t\in [0,T]}$. Moreover,
	 	\[Y_t^\mu-\ell_t=\int_0^tb(Y_s^\mu,\mu_s,\La_s)\d s+\int_0^t\sigma(Y_s^\mu,\mu_s,\La_s)\d W_t\in \mathcal{C}([0,T];\R^d).\]
	 	Hence, $(Y_t^\mu)\in \mathscr{Y}_T(\ell)$.
	 	Now, let us introduce a mapping
	 	\begin{equation*}
	 		\Psi:\mathscr{Y}_T(\ell)\to \mathscr{Y}_T(\ell), \ X\mapsto Y^\mu.
	 	\end{equation*}
	 	%Here, notice that for $\omega\in \Omega^0$, by \eqref{c-8},
	 	%\begin{equation}
	 	%\begin{split}
	 	%Y_t^\mu&=Y_0+\int_0^t\!b(Y_s^\mu,\mu_s(\omega),\La_s(\omega))\d s +\int_0^t\sigma(Y_s^\mu,\mu_s(\omega),\La_s(\omega))\d W_s\\
	 	%&\quad +\int_0^t\!\int_{\R^d}\!g(Y_s^\mu,\mu_s(\omega^0),\La_s(\omega^0), z)\wt{\mathcal{N}}(\d s,\d z),
	 	%\end{split}
	 	%\end{equation} which yields immediately that $\law(Y^\mu|\F^0)$ is in $\mathscr{Y}_T$.
	 	In what follows, we shall show $\Psi$ is a strict contraction in $(\mathscr{Y}_T(\ell), {\bm \rho}_{2,\infty})$.
	 	
	 	Indeed, for another $\tilde X\in \mathscr{Y}_T(\ell)$, let $(Y_t^\nu)$ be a solution to SDE \eqref{c-8} by replacing $\mu_t=\law(X_t|\F_t^0)$ there with $\nu_t:=\law(\tilde X_t|\F_t^0)$ for $t\in [0,T]$.
	 	Applying It\^o's formula, by (A1) and Burkholder-Davis-Gandy's inequality,
	 	\begin{align*}
	 		\E\big[\sup_{0\leq s\leq t}\! |Y_s^\mu-Y_s^\nu|^2\big]
	 		&\leq 2T\E\Big[\int_0^t\! |b(Y_s^\mu,\mu_s,\La_s)-b(Y_s^\nu,\nu_s,\La_s)|^2\d s\Big]\\
	 		&\quad +2\E\Big[\Big(\int_0^t(\sigma(Y_s^\mu,\mu_s,\La_s)-\sigma(Y_s^\nu,\nu_s ,\La_s))\d W_s\Big)^2\Big]\\
	 		&\leq ( 2T K_1 \!+\! 2C_2K_1)\E\Big[\int_0^t\!|Y_s^\mu\!-\!Y_s^\nu|^2\!+\!\W_2(\mu_s,\nu_s)^2 \d s\Big]  \\
	 		&\leq \wt C\!\int_0^t\!\E[\sup_{0\leq r\leq s}|Y_r^\mu\!-\!Y_r^\nu|^2]\d s\!+\!\wt C \!\int_0^t\!\E[\sup_{0\leq r\leq s}|X_r\!-\!\tilde X_r|^2]\d s,
	 	\end{align*} where $\wt C =2TK_1+2C_2 K_1$. By virtue of Gronwall's inequality,
	 	\begin{equation}\label{c-9}
	 		\begin{aligned}
	 			{\bm \rho}_{t,2,\infty}(Y,\tilde Y):=\E[\sup_{0\leq s\leq t}\!|Y_s^\mu\!-\!Y_s^\nu|^2]&\leq \wt C \e^{\wt C  t}\int_0^t\!\E\big[\sup_{0\leq r\leq s}\!|X_r-\tilde X_r|^2\big]\d s\\
	 			&=\wt C \e^{\wt C  t}\int_0^t\!{\bm \rho}_{s,2,\infty}(X,\tilde X)\d s.
	 		\end{aligned}
	 	\end{equation}
Iterating inequality \eqref{c-9} for $n\in \N$,
\begin{equation}\label{c-10}
\begin{split}
	 			{\bm \rho}_{2,\infty}(\Psi^n(X),\Psi^n(\tilde X))&={\bm \rho}_{T,2,\infty}(\Psi^n(X),\Psi^n(\tilde X))\\
&\leq {\wt C}^n\e^{n\wt C T}\int_0^T\frac{(T-s)^{n-1}}{(n-1)!}{\bm \rho}_{s,2,\infty}(X,\tilde X)\d s\\
	 			&\leq \frac{{\wt C}^n\e^{n\wt C T} T^n}{n!} {\bm \rho}_{T,2,\infty} (X,\tilde X)\\
&=\frac{{\wt C}^n\e^{n\wt C T} T^n}{n!}  {\bm \rho}_{2,\infty} (X,\tilde X),
\end{split}
\end{equation}
 where $\Psi^n$ stands for the $n$-th composition of the mapping $\Psi$ with itself. \eqref{c-10} shows that for $n$ large enough, $\Psi^n$ is a strict contraction. By fixed point theorem, $\Psi$ admits a unique fixed point in $(\mathscr{Y}_T(\ell), {\bm \rho}_{2,\infty})$, which satisfies \eqref{a-1}. So, SDE \eqref{a-1} admits a  solution.
 Furthermore, by virtue of Theorem \ref{thm-2.1}, SDE \eqref{a-1} admits a unique solution. Hence, we complete the proof. 
	 \end{proof}
	
\section{Averaging principle for MKV SDEs}
	
	 This section is devoted to the study on the limit behavior of two time-scale MKV SDEs with Markovian regime-switching. For $\veps>0$, let $(\La_t^\veps)_{t\geq 0}$ be a continuous-times Markov chain on $\S=\{1,2,\ldots, m_0\}$ with $m_0\leq \infty$ with the transition rate matrix $\frac 1\veps(q_{ij})_{i,j\in\S}$. When the state space $\S$ is infinite, the strong solutions and the strong Feller properties for regime-switching diffusion processes have been studied in \cite{Sh15}. A more comprehensive study on the properties of Markov chain with the transition rate matrix $\frac 1\veps(q_{ij})_{i,j\in\S}$ was launched in \cite{YZ98}. As we shall consider the $L^2$-convergence of the two time-scale MKV SDEs, the diffusion coefficient $\sigma(x,\mu,i)$ should not depend on the fast component $i$ (see \cite{Sh23} for more discussion on this assertion). Besides, more restrictions on the jump coefficient $g$ are needed in current work, that is, $g$ is supposed also independent of $x$, $\mu$ and $i$.  In brief, in this section we consider the following condition distribution dependent SDE:
\begin{equation}\label{d-1}
  \begin{split}
	\d Y_t^\veps&= b(Y_t^\veps,\law(Y_t^\veps|\F_t^0), \La_t^\veps)\d t \!+\!\sigma(Y_t^\veps,\law(Y_t^\veps|\F_t^0))\d W_t\!  + \!\int_{\R^d_0}\!g(z) \wt{\mathcal{N}}(\d t,\d z),
  \end{split}
\end{equation} endowed with the initial value $Y_0^\veps=Y_0\in \F_0$ satisfying $\E|Y_0|^2<\infty$, $\La_0^\veps=i_0\in\S$. Recall that $(\pi_i)$ is the invariant probability measure given in (A4). Define
\begin{equation}\label{d-2}
\begin{split}
	\bar{b}(x,\mu)&= \sum_{i\in\S}b(x,\mu,i)\pi_i, \quad x,z\in \R^d,\, \mu\in\! \pb(\R^d).
\end{split}
\end{equation} Then the limit process   will be given by the following distribution dependent SDE
\begin{equation}\label{d-3}
 \d \bar{Y}_t=\bar{b}(\bar{Y}_t,\law(\bar{Y}_t))\d t+ {\sigma}(\bar{Y}_t,\law(\bar Y_t))\d W_t+\int_{\R^d_0}\! {g}(z)\wt{\mathcal{N}}(\d t,\d z)
\end{equation} with $\bar{Y}_0=Y_0$ given as above.

%In order to ensure the second moments of the processes $(Y_t^\veps)_{t\geq 0}$ and $(\bar Y_t)_{t\geq 0}$ are finite for all $t\geq 0$, we introduce the following Foster-Lyapunov condition.
%\begin{itemize}
%  \item[$\mathrm{(A5)}$] There exist  a smooth function $G:\R^d\to (0,\infty)$ and  constant $\gamma_0,\,\gamma_1, \,c_0>0$ such that
%      $G(x)\geq c_0|x|^2$ for all $x\in \R^d$ and
%      \begin{align*}
%      \mathscr{L}_{i,\mu} G(x)&\!:= \!\sum_{k=1}^d\! b_k(x,\mu,i)\partial_{k}G(x)\!+\!\frac 12\sum_{k,j=1}^d\! a_{kj}(x,\mu,i) \partial_{kj}^2 G(x)\!+\!\!\int_{\R^d_0}\!\!\!G(x\!+\!g(z))\!-\! G(x)\lambda(\d z)
%      \\
%      &\leq -\gamma_0 G(x)+\gamma_1,\qquad\qquad x\in \R^d,\ i\in \S,\ \mu\in \pb(\R^d),
%      \end{align*}    and $a(x,\mu,i):=(\sigma\sigma^\ast)(x,\mu,i)$.
%\end{itemize}
	
\begin{mythm}\label{thm-3.1}
Assume that (A1)-(A4) hold. In addition, suppose that $\sigma(x,\mu,i)$ is independent of argument $i$ and $g(x,\mu,i,z)$ independent of arguments $x,\mu,i$. Suppose that there exists a function $h:\R^d\to (0,\infty)$ such that
	 	\[ |g(z)|\leq h(z), \ \forall\,z\in \R^d, \ \text{and}\  \int_{\R_0^d} h(z)^2\lambda(\d z)<\infty.\]
	 	Then,
	 	\begin{equation}\label{d-4}\lim_{\veps\to 0} \E\big[|Y_t^\veps-\bar{Y}_t|^2\big]=0.
	 	\end{equation}
	 \end{mythm}

	 We need some preparation before presenting the proof of this theorem.
	
\begin{mylem}\label{lem-3.1}
Under the assumptions of Theorem \ref{thm-3.1}, for $t>s\geq 0$, it holds
 \begin{equation}\label{d-5}
 \begin{split}
 \E\big[\W_2^2(\law(Y_t^\veps|\F_t^0),\law(\bar{Y}_t))\big]&\leq \E[ |Y_t^\veps-\bar{Y}_t|^2],\\
 \W_2^2\big(\law(Y_t^\veps|\F_t^0),\law(Y_s^\veps|\F_s^0)\big)&\leq \E^1\big[\big|Y_t^\veps-Y_s^\veps|^2\big],
 \end{split}
 \end{equation}
 and
 \begin{equation}\label{d-5.5}
 \E\big[|Y_t^\veps\!-\!Y_s^\veps|^2\big|\F_s^0\big]\leq \widehat{C}\int_s^t\!\E\big[ 1+|Y_r^\veps|^2\!+\!\E^1[|Y_r^\veps|^2]\big|\F_s^0\big]\d r\!+\!(t\!-\!s)\int_{\R_0^d}\!h(z)^2\lambda(\d z),
 \end{equation}
 where $\widehat{C}=9K_2(1+t-s)$. Similar estimates are also valid for the process $(\bar Y_t)$.
\end{mylem}
	
	 \begin{proof}
	 	Notice that according to \cite[Lemma 2.4,p.113]{Car},  the mapping $\omega^0\mapsto \law(Y^\veps_t(\omega^0, \cdot))$ from $(\Omega^0,\F^0,\p^0)$ to $\pb(\R^d)$ is almost surely well defined under $\p^0$, and  provides a conditional law of $Y^\veps_t$ given $\F^0_t$. Also, notice that
for $\omega^0\in \Omega^0$,
\[\p^1[Y_t^\veps(\omega^0,\cdot)\in D]=[\law(Y_t^\veps|\F_t^0)(\omega^0)](D),\quad D\in \B(\R^d),\]
	 	and the process $(\bar{Y}_t)$ does not depend on $\Omega^0$. So, under $\p^1$, the distribution of $\bar{Y}_t$ is $\law(\bar{Y}_t)$. Hence, by the definition of Wasserstein distance,	
\begin{gather*}
\W_2^2(\law(Y_t^\veps|\F_t^0)(\omega^0),\law(\bar{Y}_t))\leq \E^1\big[|Y_t^\veps(\omega^0,\cdot)-\bar{Y}_t|^2\big],\\
\W_2^2(\law(Y_t^\veps|\F_t^0)(\omega^0),\law(Y_s^\veps|\F_s^0)(\omega^0) )\leq \E^1\big[|Y_t^\veps(\omega^0,\cdot)-Y_s^\veps(\omega^0,\cdot)|^2\big]
\end{gather*}
Hence, the desired estimate \eqref{d-5} follows immediately.
% by taking expectation w.r.t. $\p^0$ on both sides of the previous inequality.

By growth conditions for $b,\sigma$ and $g$, it yields from \eqref{d-1} that
\begin{align*}
  &\E[|Y_t^\veps-Y_s^\veps|^2\big|\F_s^0\big]\\
  &\leq 9\E\Big[\Big|\int_s^t\!b(Y_r^\veps,\law(Y_r^\veps|\F_r^0),\La_r^\veps) \d r\Big|^2+\int_s^t\|\sigma(Y_r^\veps,\law(Y_r^\veps|\F_r^0))\|^2\d r\\
  &\qquad \qquad+\int_s^t\!\int_{\R^d_0}\!h(z)^2\lambda(\d z)\d r\Big|\F_s^0\Big]\\
  &\leq 9(1+t-s)K_2\int_s^t\!\E\big[1+|Y_r^\veps|^2+\E^1[|Y_r^\veps|^2] \big|\F_s^0\big]\d r\!+\!(t\!-\! s)\int_{\R_0^d}\!h(z)^2\lambda(\d z).
\end{align*}
This is just \eqref{d-5.5} and hence we complete the proof.
\end{proof}

\noindent\textbf{Proof of Theorem \ref{thm-3.1}}\
According to Theorem \ref{thm-2.2}, under the assumptions of this theorem, SDE  \eqref{d-1}   admits a unique strong solution $(Y_t^\veps)_{t\geq 0}$  for each $\veps>0$. By viewing SDE \eqref{d-3} as a special case of SDE \eqref{d-1}, we have that SDE \eqref{d-3} also admits a unique strong solution  $(\bar Y_t)$.

Using condition (A2), which is uniform in $i\in\S$, there exists a positive constant    $C_1 $ depending on $T$  such that
 \begin{equation}
 \label{d-6.0}
 \E\big[\sup\nolimits_{t\in [0,T]} |Y_t^\veps|^2\big]\leq   C_1,\quad \E\big[\sup\nolimits_{t\in [0,T]}
 |\bar{Y}_t|^2\big] \leq C_1.
 \end{equation}
 %Indeed, it follows from It\^o's formula that
% \begin{align*}
% \E G(Y_t^\veps)-\E G(Y_s^\veps)&= \E\Big[\int_s^t\mathscr{L}_{\La_r^\veps,\mu_r} G(Y_r^\veps)\d r\Big]\\
% &\leq \E\Big[\int_s^t-\gamma_0 G(Y_r^\veps)+\gamma_1\d r\Big],\quad \forall\,0\leq s<t.
% \end{align*}
% Then, Gronwall's inequality and (A5) yield
% \[c_0\E|Y_t^\veps|^2\leq \E G(Y_t^\veps)\leq \E G(Y_0^\veps)\e^{-\gamma_0 t }+\frac{\gamma_1}{\gamma_0}\big(1-\e^{-\gamma_0t }\big),\quad t\geq 0.\]
%The fact \eqref{d-6.0}  will be used in the deduction below and is no longer mentioned. For the sake of notation simplicity, we put $m_2(\mu)=\int_{\R^d}|x|^2\mu(\d x)$ for $\mu\in \pb(\R^d)$.

We need to apply the time discretization method as in \cite{Ve91,Liu}. For this, let $\delta\in (0,1)$ and $s(\delta)=[s/\delta]\delta$ for $s\!\geq \!0$, where $[s/\delta]\!=\!\max\{k\in\Z; k\leq s/\delta\}$. Applying It\^o's formula, for any  $T\!>\!0$,
\begin{equation}\label{d-6}
\begin{aligned}
&\E\big[\sup_{t\in [0,T]}|Y_t^\veps\!-\!\bar Y_t|^2\big]\\
&\leq C\E\Big[\int_0^T\!\|\sigma(Y_s^\veps,\law(Y_s^\veps|\F_s^0))- \sigma(\bar Y_s,\law(\bar{Y}))\|^2\d s\Big]\\
&\quad +\! \E\Big[\!\max_{t\in [0,T]}\!\int_0^t\!\!2\la Y_s^\veps\!-\!\bar Y_s,b(Y_s^\veps,\law(Y_s^\veps|\F_s^0),\La_s^\veps)\!-\!\bar b(\bar Y_s,\law(\bar Y_s))\raa \d s\Big]\\
&\leq C \E\Big[\int_0^t\!\|\sigma(Y_s^\veps,\law(Y_s^\veps|\F_s^0))- \sigma(\bar Y_s,\law(\bar{Y}))\|^2\d s\Big]\\
&\quad +T\E\int_0^T|b(Y_s^\veps,\law(Y_s^\veps|\F_s^0),\La_s^\veps)- b(Y_{s(\delta)}^\veps,\law(Y_{s(\delta)}^\veps|\F_{s(\delta)}^0), \La_s^\veps)|^2\d s\\
&\quad +T \E\int_0^T|\bar{b}(\bar Y_{s(\delta)},\law(\bar Y_{s(\delta)}))-\bar b(\bar Y_s,\law(\bar Y_s))|^2\d s\\
&\quad +T\E\int_0^T\!|\bar b(Y_{s(\delta)}^\veps,\law(Y_{s(\delta)}^\veps|\F_{s(\delta)}^0)) -\bar{b} (\bar Y_{s(\delta)},\law(\bar Y_{s(\delta)}))|^2\d s\\
&\quad +\E\Big[\max_{t\in [0,T]}\!\Big|\int_0^t\!b(Y_{s(\delta)}^\veps, \law(Y_{s(\delta)}^\veps|\F_{s(\delta)}^0),\La_s^\veps))-\bar{b} (Y_{s(\delta)}^\veps,\law(Y_{s(\delta)}^\veps|\F_{s(\delta)}^0))\d s\Big|^2\Big]\\
&=:\mathrm{(I)}+\mathrm{(I\!I)}+\mathrm{(I\!I\!I)}+\mathrm{(I\!V)} +\mathrm{(V)}.
\end{aligned}
\end{equation}
We proceed to estimate the terms $\mathrm{(I)}$, $\mathrm{(I\!I)}$, $\mathrm{(I\!I\!I)}$, $\mathrm{(I\!V)}$, $\mathrm{(V)}$ one by one.
	
By (A1) and Lemma \ref{lem-3.1},
\begin{equation}\label{d-6.1}
\begin{split}
  \mathrm{(I)}&=C \E\Big[\int_0^T\!\|\sigma(Y_s^\veps,\law(Y_s^\veps |\F_s^0))-\sigma(\bar Y_s,\law(\bar Y_s))\|^2\d s\Big]\\
  &\leq CK_1 \E\Big[\int_0^T\!|Y_s^\veps-\bar Y_s|^2+\W_2(\law (Y_s^\veps|\F_s^0),\law(\bar Y_s))^2\d s\Big]\\
  &\leq 2CK_1\int_0^T\!\E\big[|Y_s^\veps-\bar Y_s|^2\big]\d s.
\end{split}
\end{equation}
Analogously, by  (A1) and Lemma \ref{lem-3.1},
\begin{equation}\label{d-6.2}
\begin{aligned}
  \mathrm{(I\!I)}&\leq C\E\Big[\int_0^T\!|Y_s^\veps-Y_{s(\delta)}^\veps|^2+ \E^1[|Y_s^\veps-Y_{s(\delta)}^\veps|^2]\d s\Big]\\
  &\leq C\int_0^T\E|Y_s^\veps-Y_{s(\delta)}^\veps|^2 \d s\leq C\delta.\\
  \mathrm{(I\!I\!I)}&\leq C\int_0^T\E|\bar Y_s-\bar Y_{s(\delta)}|^2\d s\leq C\delta.\\
  \mathrm{(I\!V)}&\leq C\int_0^T\E[|Y_{s(\delta)}^\veps-\bar Y_{s(\delta)}|^2] \d s\leq C\!\int_0^T\!\E[\sup_{r\in [0,s]} |Y_r^\veps -\bar Y_r|^2]\d s.
\end{aligned}
\end{equation}
The estimate of $\mathrm{(V)}$ is a little complicated.
First, by (A2) and \eqref{d-6.0}, we have
\begin{equation}\label{d-7}
\begin{split}
  &\E\Big|\int_{T(\delta)}^T\!\! \big(b(Y_{s(\delta)}^\veps,\law(Y_{s(\delta)}^\veps|\F_{s(\delta)}^0), \La_s^\veps)-\bar b(Y_{s(\delta)}^\veps,\law(Y_{s(\delta)}^\veps|\F_{s(\delta)}^0)) \big)\d s\Big|^2\\
  &\leq 4\delta K_2\int_{T(\delta)}^T\!\! (1+2\E|Y_{s(\delta)}^\veps|^2) \d s\leq C(T)\delta^2.
\end{split}
\end{equation}
Next, for each $k\in \Z_+$,
\begin{equation*}
\begin{aligned}
  &\E\Big|\int_{k\delta}^{(k+1)\delta}\!\!\big( b(Y_{s(\delta)}^\veps,\law(Y_{s(\delta)}^\veps|\F_{s(\delta)}^0), \La_s^\veps)-\bar b(Y_{k\delta}^\veps,\law(Y_{k\delta}^\veps|\F_{k\delta}^0))\big) \d s\Big|^2\\
  &=2\E\Big[\int_{k\delta}^{(k+1)\delta}\!\!\int_{r}^{(k+1)\delta}\!\! \big( b(Y_{k\delta}^\veps,\law(Y_{k\delta}^\veps|\F_{k\delta}^0), \La_s^\veps)-\bar b(Y_{k\delta}^\veps,\law(Y_{k\delta}^\veps|\F_{k\delta}^0))\big)\\
  &\quad \cdot\big(b(Y_{r(\delta)}^\veps,\law(Y_{r(\delta)}^\veps|\F_{r(\delta)}^0), \La_r^\veps)-\bar b(Y_{r(\delta)}^\veps,\law(Y_{r(\delta)}^\veps|\F_{r(\delta)}^0))\big)\d s\d r\Big]\\
  &=4\E\Big[\int_{k\delta}^{(k+1)\delta}\!\!\int_{r}^{(k+1)\delta}\!\!\E \big[ b(Y_{k\delta}^\veps,\law(Y_{k\delta}^\veps|\F_{k\delta}^0), \La_s^\veps)-\bar b(Y_{k\delta}^\veps,\law(Y_{k\delta}^\veps|\F_{k\delta}^0)) \big|\F_r\big]\\
  &\qquad\quad  \cdot \big( b(Y_{r(\delta)}^\veps,\law(Y_{r(\delta)}^\veps|\F_{r(\delta)}^0), \La_r^\veps)-\bar b(Y_{r(\delta)}^\veps,\law(Y_{r(\delta)}^\veps|\F_{r(\delta)}^0))\big) \d s \d r\Big]\\
  &\leq 4 \E\Big[\int_{k\delta}^{(k+1) \delta}\!\!\int_{r}^{(k+1)\delta}\!\!\big|
  P_{\frac{s-r}{\veps}} b(Y_{k\delta}^\veps,\law(Y_{k\delta}^\veps|\F_{k\delta}^0),\cdot) (\La_r^\veps)-\pi\big(b(Y_{k\delta}^\veps,\law(Y_{k\delta}^\veps |\F_{ k\delta}^0),\cdot)\big)\big|\\
  &\qquad\qquad \quad \cdot \big(|b(Y_{k\delta}^\veps,\law(Y_{k\delta}^\veps|\F_{k\delta}^0),\La_s^\veps)|+
  |\bar{b}(Y_{k\delta}^\veps,\law(Y_{k\delta}^\veps|\F_{k\delta}^0))|\big)
  \d s\d r\Big].
\end{aligned}
  \end{equation*}
By virtue of (A4) and (A2), the previous inequality yields
\begin{equation}\label{d-8}
\begin{split}
&\E\Big|\int_{k\delta}^{(k+1)\delta}\!\!\big( b(Y_{s(\delta)}^\veps,\law(Y_{s(\delta)}^\veps|\F_{s(\delta)}^0), \La_s^\veps)-\bar b(Y_{k\delta}^\veps,\law(Y_{k\delta}^\veps|\F_{k\delta}^0))\big) \d s\Big|^2\\
&\leq 2\E\Big[\int_{k\delta}^{(k+1)\delta}\!\! \int_r^{(k+1)\delta}\!\!\max_{i\in\S} |b(Y_{k\delta}^\veps,\law(Y_{k\delta}^\veps|\F_{k\delta}^0),i)| \|P_{\frac{ s-r}{\veps}}(\La_r^\veps,\cdot)-\pi\|_{\mathrm{var}}\\
&\qquad \qquad\qquad  \qquad \cdot 2\sqrt{K_2(1+|Y_{k\delta}^\veps|^2+\E^1[|Y_{k\delta}^\veps|^2])}\d s \d r\Big]\\
&\leq 4K_2\E\big[ 1+2|Y_{k\delta}^{\veps}|^2\big]\int_{k\delta}^{(k+1)\delta}\!\!\int_r^{(k+1)\delta}\!\! \eta \e^{-\beta\frac{s-r}{\veps}}\d s\d r\\
&\leq 4K_2\big(1+\sup_{t\in [0,T]}\E\big[|Y_t^\veps|^2\big]\big)\Big(\frac{\veps\delta}{\beta}-\frac{\veps^2} {\beta^2}+\frac{\veps^2} {\beta^2}\e^{-\beta\frac{\delta}{\veps}} \Big).
\end{split}
\end{equation}
By virtue of \eqref{d-7}, \eqref{d-8} and \eqref{d-6.0},
\begin{equation}\label{d-9}
\begin{split}
  \mathrm{(V)}&=\E\max_{t\in [0,T]} \Big|\!\int_0^t\!b( Y_{s(\delta)}^\veps,\law(Y_{s(\delta)}^\veps |\F_{s(\delta)}^0), \La_s^\veps)-\bar{b}(Y_{s(\delta)}^\veps,\law(Y_{s(\delta)}^\veps|\F_{ s(\delta)}^0)) \d s\Big|^2\\
  &\leq 2\Big[\frac T\delta\Big] \sum_{k=0}^{[T/\delta]-1}\!\E\Big| \int_{k\delta}^{(k+1)\delta}\!\! \big(b(Y_{k\delta}^\veps, \law( Y_{k\delta}^\veps|\F_{k\delta}^0), \La_s^\veps)-\bar{b}( Y_{k\delta}^\veps,\law(Y_{k\delta}^\veps|\F_{k\delta}^0))\big)\d s\Big|^2\\
  &\quad +2\E\Big|\int_{T(\delta)}^T\!b( Y_{s(\delta)}^\veps,\law(Y_{s(\delta)}^\veps |\F_{s(\delta)}^0), \La_s^\veps)-\bar{b}(Y_{s(\delta)}^\veps,\law(Y_{s(\delta)}^\veps|\F_{ s(\delta)}^0)) \d s\Big|^2\\
  &\leq  C(T) \Big(\frac{\veps}{\beta\delta}-\frac{\veps^2} {\beta^2\delta^2}+\frac{\veps^2} {\beta^2\delta^2}\e^{-\beta\frac{\delta}{\veps}} \Big)+C(T)\delta^2.
\end{split}
\end{equation}

In all, inserting the estimates \eqref{d-6.1}, \eqref{d-6.2}, \eqref{d-9} into \eqref{d-6}, we obtain
\begin{equation}\label{d-10}
  \begin{split}
    &\E\big[ \sup_{t\in [0,T]}|Y_t^\veps-\bar{Y}_t|^2\big]\\
    &\leq C\int_0^T\!\E\big[\sup_{s\in [0,t]}\!|Y_s^\veps-\bar{Y}_s|^2\big] \d t + C (\delta+\delta^2)+C(T) \Big(\frac{\veps}{\beta\delta}-\frac{\veps^2} {\beta^2\delta^2}+\frac{\veps^2} {\beta^2\delta^2}\e^{-\beta\frac{\delta}{\veps}} \Big).
  \end{split}
\end{equation}
Take $\delta>0$ such that $\lim_{\veps\to 0}\frac{\veps}{\delta}=0$, then we get from \eqref{d-10} by Gronwall's inequality that
\[\lim_{\veps\to 0}\E[|Y_t^\veps-\bar{Y}_t|^2]=0,\]
which completes the proof of this theorem.
\fin
	
	 \section{Conditional propagation of chaos}
	
	 In this part we are interested in the following particle system with mean-field interaction and Markovian regime-switching driven by L\'evy processes:
	 \begin{equation}\label{e-1}
	 	\begin{split}
	 		X_t^{k,N}&=X_0^{k}+\int_0^t\!b(X_s^{k,N}, \mu_s^N,\La_s )\d s\!+\!\int_0^t\!\sigma(X_s^{k,N}, \mu_s^N,\La_s )\d W_s^k\\
	 		&\qquad\quad\ \, +\!\int_0^t\!\int_{\R^d}\!g(X_{s-}^{k,N}, \mu_{s-}^N, \La_{s-} , z)\wt{\mathcal{N}}^k(\d s,\d z),
	 	\end{split}
	 \end{equation}
	 where
	 \[\mu_t^N=\frac 1N\sum_{k=1}^N\delta_{X_t^{k,N}},\quad \text{where $\delta_x$ denotes the Dirac measure at point $x$.}\]
	 For $k\geq 1$, $(W_t^k)$ is a $d$-dimension standard Brownian motion, $\mathcal{N}^k(\d t,\d z)$ is a Poisson random measure with intensity measure $\lambda(\d z)\d t$ satisfying $\int_{\R^d_0} (1\wedge |z|^2)\lambda(\d z)<\infty$, and $\wt{\mathcal{N}}^k(\d t,\d z)=\mathcal{N}^k(\d t,\d z)-\lambda(\d z)\d t$ stands for the associated compensated Poisson random measure.  $(\La_t)$ is a continuous-time Markov chain on $\S=\{1,2,\ldots,m_0\}$ with $m_0\leq \infty$ and transition rate matrix $  (q_{ij})_{i,j\in\S}$. $\La_0=i_0\in\S$.  As mentioned in the introduction,  $(W_t^k)$, $ \mathcal{N}^k(\d t,\d z) $, $k\geq 1$, and $(\La_t)$ are mutually independent, and   $(W_t^k)$, $\mathcal{N}^k(\d t,\d z)$ are defined on $(\Omega^1,\F^1,\p^1)$ and $(\La_t)$ is defined on $(\Omega^0,\F^0,\p^0)$. Assume the initial value $\{X_0^k\}_{k\geq 1}$ is a sequence of i.i.d. $\F_0^1$-measurable random variables defined on $(\Omega^1,\F^1,\p^1)$.
	 %Notice that in current work $\S$ is allowed to be an infinitely countable state space.

	 In order to study the limit behavior of the system \eqref{e-1}, we introduce an auxiliary system as follows: for $k\geq 1$,
	 \begin{equation}\label{e-2}
	 	\begin{split}
	 		\hat{X}_t^k&= X_0^k+\int_0^t\!b(\hat{X}_s^k,\law(\hat{X}_s^k|\F_s^0), {\La}_s ) \d s\!+\!\int_0^t\!\sigma(\hat{X}_s^k,\law(\hat{X}_s^k|\F_s^0), {\La}_s )\d W_s^{k}\\
	 		&\qquad +\int_0^t\int_{\R^d_0}\! g(\hat{X}_{s-}^k, \law(\hat{X}_{s-}^k|\F_{s-}^0), {\La}_{s-} ,z)\wt{\mathcal{N}}^k(\d s,\d z).
	 	\end{split}
	 \end{equation}
	 Notice that given $\omega^0\!\in\! \Omega^0$, $\{\hat{X}_t^k(\omega^0,\cdot);k\geq 1\}$ are i.i.d.\! for $t\geq 0$ due to the uniqueness of solution to SDE \eqref{e-2} under conditions (A1), (A2).
	
	 \begin{mythm}\label{thm-4.1}
	 	Assume that (A1)-(A3) hold and for some $\E[|X_0^1|^q] <\infty$ for some $q>2$. Let $(Y_t)$ be the solution to \eqref{a-1} with initial value $Y_0=X_0^1$. Then for $T>0$ there exists a constant $\wt C>0$ depending on $T$, $d$, $q$ and $\E[|X_0^1|^q]$ such that
	 	\begin{equation}\label{e-3}
	 		\max_{1\leq k\leq N}\E\Big[\sup_{t\in [0,T]}\!|X_t^{k,N}-\hat{X}_t^k|^2\Big]\leq \wt C \epsilon_N,
	 	\end{equation}
	 	and
	 	\begin{equation}\label{e-4}
	 		\sup_{t\in [0,T]}\E\Big[\W_2^2\Big(\frac1N\sum_{k=1}^N\delta_{X_t^{k,N}}, \law(Y_t|\F_t^0)\Big)\Big]\leq \wt C \epsilon_N,
	 	\end{equation}
	 	where
	 	\begin{equation}\label{e-5}
	 		\epsilon_N=\begin{cases}
	 			N^{-1/2}+N^{-(q-2)/q},&\text{if $d<4$, $q\neq 4$},\\
	 			N^{-1/2}\log(1+N)+N^{-(q-2)/q}, &\text{if $d=4$, $q\neq 4$},\\
	 			N^{-2/d}+N^{-(q-2)/q}, &\text{if $d>4$, $q\neq d/(d-2)$}.
	 		\end{cases}
	 	\end{equation}
	 \end{mythm}
	
	 \begin{proof}
	 	Applying Burkholder-Davis-Gundy's inequality and (A1),
	 	\begin{equation}\label{e-6}
	 		\begin{aligned}
	 			&\E\big[\sup_{0\leq s\leq t}|X_s^{k,N}\!-\! \hat{X}_s^k|^2\big]\\
	 			&\leq 3\E\Big[ \Big|\int_0^t\!|b(X_s^{k,N},\mu_s^N,\La_s)\!-\!b(\hat{X}_s^k, \law(\hat{X}_s^k|\F_s^0),\La_s)|\d s\Big|^2\Big]\\
	 			&\quad +\!3\E\Big[\sup_{0\leq s\leq t}\!\Big(\!\int_0^s\!\!\big( \sigma(X_r^{k,N},\mu_r^N,\La_r)\!-\!\sigma(\hat{X}_r^k, \law(\hat{X}_r^k|\F_r^0),\La_r)\big)\d W_r^k\Big)^2\Big]\\
	 			&\quad +\!3\E\Big[\sup_{0\leq s\leq t}\!\Big(\!\int_0^s\!\int_{\R^d_0}\!\! \big(g(X_r^{k,N},\mu_r^N,\La_r)\!-\!g(\hat{X}_r^k,\law(\hat{X}_r^k |\F_r^0),\La_r)\big)\wt{\mathcal{N}}^k(\d r,\d z)\Big)^2\Big]\\
	 			&\leq (3TK_1+3K_1C_2)\E\Big[\int_0^t\!|X_s^{k,N}-\hat{X}_s^k|^2+
	 			\!\W_2^2(\mu_s^N,\law(\hat{X}_s^k|\F_s^0))\d s\Big].
	 		\end{aligned}
	 	\end{equation}
	 	Since $\frac1N \sum_{j=1}^N \delta_{(X_s^{j,N},\hat{X}_s^j)}$ is a coupling of $\mu_s^N$ and $\frac 1N\sum_{j=1}^N \delta_{\hat{X}_s^j}$, by the definition of $L^2$-Wasserstein distance and the triangle inequality,
	 	\begin{equation}\label{e-7}
	 		\begin{aligned}
	 			\W_2^2(\mu_s^N, \law(\hat{X}^k_s|\F_s^0))&\leq 2\W_2^2\big (\mu_s^N, \frac 1N \sum_{j=1}^N\delta_{\hat{X}_s^j}\big)+2\W_2^2\big( \frac 1N \sum_{j=1}^N\delta_{\hat{X}_s^j}, \law(\hat{X}_s^k|\F_s^0)\big)\\
	 			&\leq
	 			\frac 2N\sum_{k=1}^N\!|X_s^{k,N}-\hat{X}_s^k|^2  +2\W_2^2\big( \frac 1N \sum_{j=1}^N\delta_{\hat{X}_s^j}, \law(\hat{X}_s^k|\F_s^0)\big).
	 		\end{aligned}
	 	\end{equation}
	 	Using \cite[Theorem 1]{FG14}, there exists a constant $C$ depending only on $d$, $q$ such that
	 	\begin{equation}\label{e-8}
	 		\E^1\Big[\W_2^2\big(\frac 1N\sum_{k=1}^N\!\delta_{\hat{X}_s^k(\omega^0,\cdot)}, \law(\hat{X}_s^k|\F_s^0)(\omega^0)\big)\Big]\leq   C \epsilon_N.
	 	\end{equation}
	 	
	 	By the symmetry in the structure of the system of particles \eqref{e-1} (see \cite[Theorem 2.12]{Car} for details), it follows from \eqref{e-2}, \eqref{e-8} that
	 	\[\E\big[\sup_{0\leq s\leq t} |X_s^{1,N}-\hat{X}^1_s|^2\big]\leq C\int_0^t\!\Big(\E\big[|X_s^{1,N}-\hat{X}_s^1|^2\big]+  \epsilon_N\Big) \d s\]
	 	Then, Gronwall's inequality yields that
	 	\[\E\big[\sup_{0\leq s\leq t}|X_s^{1,N}-\hat{X}_s^1|^2\big]\leq \wt C \epsilon_N
	 	\] for some constant $\wt C$ depending only on $T,\,d,q$.
	 	Plugging this bound and \eqref{e-8} into \eqref{e-7}, we get a similar bound for
	 	\[\sup_{t\in [0,T]}\E\Big[\W_2^2\Big(\frac1N\sum_{k=1}^N\delta_{X_t^{k,N}}, \law(Y_t|\F_t^0)\Big)\Big],\]
	 	and complete the proof.
	 \end{proof}

	
	
	

\end{document}